\documentclass[a4paper,12pt,reqno,usenames,dvipsnames]{amsart}
\usepackage[foot]{amsaddr}
\usepackage{amsmath}
\usepackage{amsthm, enumerate, esint}
\usepackage{xurl} 
\usepackage{mathtools}
\usepackage{hyperref,url}
\usepackage{graphicx}
\usepackage{comment}
\usepackage{amssymb}
\usepackage{hyperref}
\usepackage{appendix}
\usepackage[english]{babel}

\usepackage{fancyhdr}

\usepackage{calc}
\usepackage{url}

\usepackage[text={6in,8.6in},centering]{geometry}
\usepackage{bigints}

\usepackage{srcltx}

\usepackage[T1]{fontenc}

\usepackage[usenames,dvipsnames]{color}

\theoremstyle{plain}
\newtheorem{theorem}{Theorem}[section]
\theoremstyle{plain}
\newtheorem{lemma}{Lemma}[section]

\newtheorem{definition}{Definition}[section]
\newtheorem{remark}{Remark}[section]

\renewcommand{\(}{\left(}
\renewcommand{\)}{\right)}
\renewcommand{\[}{\left[}
\renewcommand{\]}{\right]}

\newcommand{\be} {\begin{equation}}
	\newcommand{\ee} {\end{equation}}
\newcommand{\bea} {\begin{eqnarray}}
	\newcommand{\eea} {\end{eqnarray}}
\newcommand{\Bea} {\begin{eqnarray*}}
	\newcommand{\Eea} {\end{eqnarray*}}

\newcommand{\de} {\delta}
\newcommand{\grad}{\nabla}

\newcommand{\om} {\Omega}

\newcommand{\la} {\lambda}

\newcommand{\no} {\nonumber}
\newcommand{\noi} {\noindent}

\newcommand{\R}{\mathbb R}

\newcommand{\Rn}{\mathbb R^N}

\newcommand{\abs}[1]{\left\vert#1\right\vert}
\newcommand{\X}{\mathcal{X}_0^{1,2}}
\newcommand{\DX}{(\X)'}

\catcode`\@=11
\newcommand{\rnn}{\mathbb{R}^N}
\newcommand{\rnnn}{\iint_{\mathbb{R}^{2N}}}
\newcommand{\So}{\mathcal{S}_N}

\newcommand{\authorfootnotes}{\renewcommand\thefootnote{\@fnsymbol\c@footnote}}%
\allowdisplaybreaks

\numberwithin{equation}{section} \allowdisplaybreaks

\begin{document}
	\title[Global compactness result]{Global Compactness Result for a Brezis--Nirenberg-Type Problem Involving Mixed Local Nonlocal Operator}

\date{}
    \author[S. Chakraborty]{Souptik Chakraborty\textsuperscript{1}}
\email{soupchak9492@gmail.com}

    \author[D. Gupta]{Diksha Gupta\textsuperscript{1}}
\email{dikshagupta1232@gmail.com}

    \author[S. Malhotra]{Shammi Malhotra\textsuperscript{1}}
\email{shammi22malhotra@gmail.com}

    \author[K. Sreenadh]{Konijeti Sreenadh\textsuperscript{1}}
\address{\textsuperscript{1}Department of Mathematics, Indian Institute of Technology Delhi, Hauz Khas, Delhi-100016, India}
\email{sreenadh@maths.iitd.ac.in}

	
	
	\begin{abstract}
    This paper investigates the profile decomposition of Palais-Smale sequences associated with a Brezis-Nirenberg type problem involving a combination of mixed local nonlocal operators, given by  
    \begin{equation*}
    \left\{\begin{aligned}
        &-\Delta u + (-\Delta)^s u - \lambda u = |u|^{2^*-2}u \;\;\mbox{ in } \Omega,\\
        &\quad u=0\,\mbox{ in }\mathbb{R}^N\setminus \Omega.
        \end{aligned}
         \right.
    \end{equation*}

\noindent
where $\Omega\subseteq \Rn$ is a smooth bounded domain with $N \geq 3$, $s\in (0,1),\,\lambda\in\mathbb{R}$ is a real parameter and $2^* = \frac{2N}{N - 2} $ denotes the critical Sobolev exponent. As an application of the derived global compactness result, we further study the existence of positive solution of the corresponding Coron-type problem \cite{coron_topology} when $\la=0$.
		
	\end{abstract}

	\maketitle
	
	\section{introduction}

    We are concerned with the profile decomposition for the Palais-Smale ($(PS)$ in short) sequences for the energy functional associated to the following Brezis-Nirenberg type problem for the mixed local-nonlocal operator with critical nonlinearity on the smooth bounded domain $\om$ of $\Rn$
\begin{equation}
\tag{$\mathcal{P}_\lambda$} \label{Maineq}
\begin{cases}
-\Delta u + (-\Delta)^s u - \lambda u = |u|^{2^*-2} u \quad \text{in } \Omega,\\
u=0\text{ in }\Rn\setminus\Omega,
\end{cases}
\end{equation} 
where $\lambda \in \mathbb{R}$ is a real parameter, $N \geq 3$, $s \in (0,1)$ is the order of the fractional Laplacian $(-\Delta)^s$, and $2^* = \frac{2N}{N - 2}$ is the critical Sobolev exponent associated with the classical Laplacian. 

     \medskip


A key challenge in analyzing problems of the form \eqref{Maineq} arises from the lack of compactness in the embedding $\X(\Omega) \hookrightarrow L^{2^*}(\Omega)$, where $\X(\Omega)$ is the function space introduced in Section~\ref{Function Space}. Hence, a thorough analysis of the
fundamental causes of this compactness failure is highly beneficial. This failure is primarily attributed to the action of a certain noncompact group $G$, often referred to as "dislocations", which acts on the function space $\X(\om)$.

From a mathematical perspective, let $\mathcal{H}$ be a general Hilbert space, and consider a bounded sequence $\{u_n\} \subset \mathcal{H}$ that converges weakly to some limit $u \in \mathcal{H}$, i.e., $u_n \rightharpoonup u$. However, if the sequence does not converge strongly, meaning $u_n \not\to u$ in $\mathcal{H}$, the difference $u_n-u$ captures the failure of compactness. The study of this sequence, which characterizes the loss of compactness, is referred to as the \textit{profile decomposition}. Without loss of generality, one may assume $u \equiv 0$ and analyze the effect of a noncompact group $G$ acting on $\mathcal{H}$ through unitary operators, which plays a pivotal role in the structure of the decomposition.
More precisely, we seek a decomposition of the form
\begin{equation}\label{AbstractProfileDecomp}
    u_n = \sum_{j \in I} g_n^{(j)} \psi_j + r_n,
\end{equation}
where $I$ is an (at most countable) index set, for each $j \in I$, the sequence $\{g_n^{(j)}\} \subset G$ and diverges within the group, while the functions $\{\psi_j\}_{j \in I} \subset \mathcal{H}$ represent the set of all \textit{profiles}  $\Psi$ associated with $\{u_n\}$ is defined as
\begin{equation*}
    \Psi = \left\{ \psi_j \in \mathcal{H} \setminus \{0\}, \, j \in I; \, \psi_j = w^{(j)} - \lim_{n \to \infty} \left( g_n^{(j)} \right)^* u_n, \, w^{(j)} \in \mathcal{H},\, g_n^{(j)} \in G \right\},
\end{equation*}
where $\cdot^*$ denotes the adjoint element. Thus, $\Psi$ consists of all possible nonzero weak limits of the sequence $\{u_n\}$. When this set is nonempty, one can attempt to extract and remove all profiles from $u_n$ by reversing the transformation induced by the group elements $g_n^{(j)}$ and then investigate the asymptotic properties of the remainder sequence $\{r_n\} \subset \mathcal{H}$. Furthermore, the decomposition in \eqref{AbstractProfileDecomp} is characterzied by the following three properties:
\begin{enumerate}
  \item \textit{Nontriviality:} Each profile $\psi_j$ is nonzero.
  \item \textit{Asymptotic Orthogonality:} The dislocations $\{g_n^{(j)}\}$ and $\{g_n^{(k)}\}$, for different $j, k \in I$, become increasingly distinct, satisfying $g_n^{(j)^*} g_n^{(k)} \rightharpoonup 0$ for $j \neq k$.
\item \textit{$G$-Weak Convergence to Zero:} The remainder $r_n$ contains no further profiles and 
$G$-weakly converges to zero as $n \to \infty$ (see \cite{Tin-Fi}, Chapter~3 for definitions and related discussions).
\end{enumerate}
 The existence and uniqueness of such decompositions are well established (see \cite{Tao10, Tin-Fi}), making them a valuable tool in understanding compactness failures. However, the challenge in applying this framework to our setting lies in characterizing the noncompact structure of the group $G$. 

Motivated by the fundamental difficulty due to lack of compactness for the Sobolev embedding $H^1_0(\Omega)\xhookrightarrow{} L^{2^*}(\Omega)$, in a seminal work, Brezis-Nirenberg \cite{Bre-Nir-CPAMS} considered the perturbed equation
\begin{equation}\label{BNP}
\begin{cases}-\Delta u-\lambda u=|u|^{2^*-2} u & \text{ in } \Omega, \\ u=0 & \text{ on } \partial \Omega,\end{cases}
\end{equation}

where $\Omega\subset \Rn$ is a smooth bounded domain, $N\geq 3$, and $\lambda$ is real parameter. The key insight was to treat the linear term $\lambda u$ as a compact perturbation of the critical nonlinearity, enabling the recovery of compactness needed to apply variational tools. Their results revealed a delicate interplay between the domain geometry, the parameter $\lambda$, and the dimension $N$, with the existence of nontrivial solutions depending critically on these factors, establishing a sharp threshold for the existence of solutions in terms of the first Dirichlet eigenvalue. Subsequently, the nonlocal counterpart had been studied by Servedei-Valdinoci \cite{servadei-valdinoci_BrezisNirenbergFractional}. For establishing the existence of a positive weak solution of \eqref{BNP}, the local compactness result plays a pivotal role. They showed that below the level $\frac{1}{N}\mathcal{S}_{N}^{N/2}$ all $(PS)$-sequences (see Definition~\ref{PS-def}) for the energy functional associated to \eqref{BNP} are pre-compact, which opens the question of classifying all the ranges where the energy functional fails to satisfy the $(PS)$ condition. Struwe \cite{Struwe-Z} provided a comprehensive resolution to the above question through what is now known as the Global Compactness Theorem or the Palais–Smale Decomposition Theorem, offering a precise characterization of the loss of compactness for the $(PS)$-sequences. Struwe~\cite{Struwe-Z} showed that every associated $(PS)$-sequence (up to a subsequence) for the energy functional corresponding to \eqref{BNP} is a sum of its weak limit and finitely many weakly interacting rescaled bubbles (which are precisely positive solutions to the critical exponent problem  \eqref{eq:problem_at_infinity} in the entire space $\Rn$). His analysis demonstrated that compactness can break down due to the existence of nontrivial entire solutions to the following  \textit{limiting} problem, which remains invariant under translations and scalings:  
\begin{equation*}
\left\{\begin{array}{l}
-\Delta u=|u|^{2^*-2} u \;\; \text{ in } \mathbb{R}^N, \\
u \in \mathcal{D}^{1,2}(\mathbb{R}^N),
\end{array}\right.
\end{equation*}  
where  
\begin{equation*}
\mathcal{D}^{1,2}(\mathbb{R}^N) \coloneq \left\{u \in L^{2^*}(\mathbb{R}^N) : |\nabla u| \in L^2(\mathbb{R}^N)\right\}.
\end{equation*}  

\medskip
Building on this foundation, Yan \cite{yan_global_compact_p_lap} and Alves \cite{Alves1,alves_global_compact_p_lap} expanded the scope of profile decompositions to the $p$-Laplacian case ($2 \leq p < N$), considering both bounded domains and the entire space. In the sublinear regime ($1 < p < 2$), Mercuri and Willem \cite{mercuri_global_compact_p_lap} established global compactness results for the $p$-Laplacian with critical Sobolev terms in smooth bounded domains within $\mathbb{R}^N$. The framework was further extended to smooth Riemannian manifolds without boundary by Saintier \cite{nicolas_global_compact_manifold}. A broader generalization emerged through the work of Sandeep and Tintarev \cite{sandeep_manifold}, who studied a profile decomposition for bounded sequences in the Sobolev space $H^{1,2}(M)$, where $M$ is a noncompact manifold with bounded geometry. These contributions have been instrumental in refining our understanding of noncompactness phenomena in Sobolev embeddings across diverse geometric and functional settings.

\smallskip

The study of global compactness results for the fractional Laplacian was initiated in \cite{gerard_ps_fractional}, where the authors provided a general characterization of the compactness defects for arbitrary bounded sequences in $\mathcal{D}^{s,2}(\mathbb{R}^N)$, without assuming these sequences to be Palais-Smale sequences. This foundational result laid the groundwork for subsequent developments in the understanding of profile decompositions in nonlocal frameworks. Later, the work in \cite{Pala-Pis-Cal} refined these findings by strengthening fractional Sobolev embeddings using Morrey norms. This advancement provided a more transparent and effective approach to deriving profile decompositions for sequences in the homogeneous Sobolev space $\dot{H}^s$, as initially developed in \cite{Tin-Fi}.  Subsequently, \cite{Pala-Pis-Non} extended the profile decomposition framework specifically to Palais-Smale sequences associated with the fractional Laplacian operator. Jaffard, in \cite{Jaffard_p_frac_laplacian_global_cpt}, extended G\'erard's results \cite{gerard_ps_fractional} to the setting of fractional $p$-Laplacian spaces. For further developments on global compactness in this context, particularly for the fractional $p$-Laplacian and for radially symmetric functions in a ball $B \subset \mathbb{R}^N$, see also \cite{brasco_globalcpt_frac_p_laplacian}.


\smallskip
Global compactness for semilinear elliptic equations with Hardy potentials and critical Sobolev terms has been extensively studied (see \cite{smets_hardy,cao_singular_hardy,jin_hardy,guo_hardy,kang_hardy,Gao_hardy,kang_hardy2,filippucci_hardy, Bhakta-souptik-pucci}). However, a comprehensive description of the associated Palais-Smale sequences remained incomplete until \cite{Li-Yuan-Guo}, where the authors established global compactness results for the $p$-Laplace problem with combined Sobolev-Hardy terms in both bounded and unbounded domains, providing a full characterization of these sequences. Furthermore, a similar Palais-Smale decomposition has been established in various mathematical settings. Notable examples include Hebey and Robert \cite{emmanuel_Paneitz} for Paneitz-type operators, and El-Hamidi-Vetois \cite{hamidi_anisotropic} for anisotropic operators. Almaraz \cite{almaraz_manifolds} investigated compactness in the context of nonlinear boundary conditions, Mazumdar \cite{saikat_polyharmonic} studied polyharmonic operators on compact manifolds. Further contributions include the work of Chernysh \cite{chernysh2024struwetypedecompositionresultweighted}, who studied critical $p$-Laplace equations of the Caffarelli-Kohn-Nirenberg type in bounded domains.  Kirchhoff-type equations have been analyzed by Lei, R\u{a}dulescu, and Zhang \cite{radulescu_kirchhoff_pd} as well as Tian and Zhang \cite{zhang_frac_kirchhoff_pd}. Additionally, profile decomposition has also been extended to time-dependent operators (see \cite{peng_time_fractional_pd} and references therein).

\smallskip

Drawing inspiration from the foundational work on profile decomposition-Struwe’s analysis for the Laplacian \cite{Struwe-Z} and Palatucci-Pisante’s work on the fractional Laplacian \cite{Pala-Pis-Non}-we aim to investigate the profile decomposition of Palais-Smale sequences associated with an energy functional that incorporates both the Laplacian and fractional Laplacian operators. Since Struwe’s seminal paper \cite{Struwe-Z}, the widely recognized Global Compactness theorem in the Sobolev space $H^1$ has become a cornerstone of analysis. It has played a significant role in proving existence results across various mathematical domains, including ground state solutions for nonlinear Schrödinger equations, Yamabe-type equations in conformal geometry, problems related to prescribing $Q$-curvature, etc. In Theorem~\ref{PS_decomposition}, we establish an analogue of Struwe’s Global Compactness result within our framework, focusing on the profile decomposition of Palais-Smale sequences associated with the energy functional $I_{\lambda, s}$ defined in \eqref{Ener-Fun}. Although the space $\X(\Omega)$ exhibits noncompact translational invariance, a major challenge in our setting arises from the absence of a common noncompact scaling group action on $\X(\Omega)$, which significantly complicates the decomposition analysis.

Now, coming back to our problem \eqref{Maineq}, we observe that the equation is variational, and the associated energy functional is given by
\begin{equation}\label{energy-functional}
   I_{\lambda,s}(u)\coloneqq \frac{1}{2}\rho(u)^2-\frac{\lambda}{2}\int_{\Rn}|u|^2\,{\rm d}x-\frac{1}{2^*}\int_{\Rn} |u|^{2^*} {\rm d}x, \text{ for all }\, u\in\X(\Omega), 
\end{equation}
 
where $\rho(\cdot)$ is defined in \eqref{norm}. Note that $I_{\lambda,s}$ is well defined in $\X (\Omega)$ and $I_{\lambda,s} \in C^1\left(\X (\Omega),\R\right)$. Moreover, if $u\in \X (\Omega)$ is a weak solution of~\eqref{Maineq}, then $u$ is a critical point of $I_{\lambda,s}$ and vice versa. One of the standard tools in variational methods is to find a minimizer of the corresponding energy functional by minimizing sequences. The following definition gives the desired compactness condition for the extremizing sequences associated to the energy functional.

\begin{definition}\label{PS-def}
    A sequence $\{ u_k \} \subset \X(\Omega)$ is said to be a Palais-Smale $(PS)$ sequence for $I_{\lambda,s}$ at level $\beta$, if $I_{\lambda,s}(u_k) \to \beta$ in $\R$ and $I'_{\lambda,s}(u_k) \to 0$ in $(\X(\Omega))'$ as $k \to \infty$. The function $I_{\lambda,s}$ is said to satisfy $(PS)$ condition at level $\beta$, if every $(PS)$ sequence for $I_{\lambda,s}$ at level $\beta$ admits a convergent subsequence.
\end{definition}

The authors in \cite{biagi} studied the Brezis-Nirenberg problem for the mixed local nonlocal operators and, among other results, established the following theorem:  
\begin{theorem}[see \cite{biagi}]\label{thm:best_constant}  
For $ s \in (0,1) $ and any open set $ \Omega \subseteq \mathbb{R}^N $, the best Sobolev constant $ \mathcal{S}_{N,s} $ in the inequality \eqref{eq:mixed_inequality} coincides with $ \So $, the optimal constant in the classical Sobolev inequality. Furthermore, $ \mathcal{S}_{N,s} $ is never attained.  
\end{theorem}

A key idea in their proof was the use of rescaled functions $ u_k(x)\coloneq k^{\frac{N-2}{2}} u(kx) $, which allowed them to diminish the effect of the nonlocal term. Specifically, the Gagliardo seminorm satisfies  
$$
[u_k]_s^2 = k^{2s-2}[u]_s^2 \to 0 \quad \text{as } k \to \infty,
$$
and a similar decay holds for the $ L^2 $-norm of $ u_k $. This, combined with the compact embedding of $\X(\Omega)$ into $L^2(\Omega)$, played a crucial role in identifying the limiting profiles in the global compactness result established by Struwe \cite{Struwe-Z}. Indeed, the profiles in that case corresponded to the solutions of problem \eqref{eq:problem_at_infinity}. This observation led us to find an answer for the question of whether the space $\X(\Omega)$ admits a compact embedding into $H^s(\Omega)$ (refer to Section~\ref{sec:notations} for the definitions of these spaces). We show that the answer is affirmative, which follows from Corollary 6.19 in \cite{giovanni} (see Lemma \ref{compact_embedding} for a detailed proof). Consequently, we landed in the case Laplacian and hence established the following global compactness result:  

\begin{theorem}\label{PS_decomposition}  
Let $ \{u_k\} \subseteq \X(\Omega) $ be a Palais-Smale sequence for the functional $ I_{\lambda,s} $ at the level $ \beta $. Then there exist $ u_0 \in \X(\Omega) $, $ l \in \mathbb{N} \cup \{0\} $, and sequences $ \{(R_k^i, x_k^i)\} $ for $ 1 \leq i \leq l $ such that, up to a subsequence (still denoted by $ \{u_k\} $), the following holds:  
\begin{eqnarray}
     &(i)& \left\lVert u_k - u_0 - \sum_{i=1}^{l} U_k^{i} \right\rVert_{\mathcal{D}^{1,2}(\rnn)} \longrightarrow 0,  \notag \\  
    &(ii)& -\Delta u_0 + (-\Delta)^s u_0 - \lambda u_0 = |u_0|^{2^*-2} u_0 \quad \text{in } \Omega, \notag\\ 
    &(iii)& -\Delta U_i = |U_i|^{2^*-2} U_i \quad \text{in } \mathbb{R}^N, \quad \text{for } 1 \leq i \leq l,  \label{eq:problem_at_infinity}  \\  
    &(iv)& \left|\log \left(\frac{R_k^i}{R_k^j}\right)\right| + \frac{|x_k^i - x_k^j|}{R_k^i} \to \infty, \quad \text{for } 1 \leq i \neq j \leq l. \notag
\end{eqnarray}

\noi
Here, $ U_k^{i}(x) := (R_k^i)^{-\frac{N-2}{2}} U_i\left(\frac{x - x_k^i}{R_k^i}\right) $ for $ 1 \leq i \leq l $ solves the corresponding problem at infinity, i.e., \eqref{eq:problem_at_infinity}. Moreover, as $k\to \infty$ we have,  
\begin{equation*}
    I_{\lambda,s}(u_k) \to I_{\lambda,s}(u_0) + \sum_{i=1}^l I_{\infty}(U_i).
\end{equation*}  
\end{theorem}  

\smallskip
As established in Theorem \ref{thm:best_constant}, the best constant in the associated Sobolev inequality is not attained; that is, ground state solutions for the critical problem under consideration do not exist. In such scenarios, profile decomposition plays a pivotal role by identifying the precise levels at which compactness is lost. It not only provides a clearer understanding of the concentration phenomena but also delineates the energy range where one must search for higher energy solutions.

Motivated by this insight, the pursuit of high energy solutions becomes a natural direction. There are two widely adopted strategies to overcome the loss of compactness: one involves introducing suitable perturbations to the equation, while the other exploits variations in the topology or geometry of the domain. The latter approach gives rise to the celebrated Coron problem, which examines how domain topology can be leveraged to recover compactness and hence ensure the existence of solutions. A vast body of literature explores these techniques using the framework of profile decomposition to prove the existence and multiplicity of solutions; see, for instance, \cite{kaniska_multiplicity_p_lap, clapp_multiple_soln, clapp_entire_nodal, sergio_multiple_soln, jerome_multiple_soln, alessio_multiple_soln_fractional, zhang_multiple_soln_sobolev_hardy}.

\smallskip
The classical Coron's problem \cite{coron_topology}, introduced in 1984, addresses critical elliptic equations in domains with certain geometric features. It asserts that if there exists a point $x_0\in\R^N$ and radii $R_2 > R_1 > 0$ satisfying 
\begin{equation*}
\left\{R_1 \leq\left|x-x_0\right| \leq R_2\right\} \subset \Omega, \quad\left\{\left|x-x_0\right| \leq R_1\right\} \not \subset \bar{\Omega}.
\end{equation*}
Then the problem
\begin{equation*}
\begin{cases}
-\Delta u=u^{\frac{N+2}{N-2}} & \text{in $\Omega$,} \\
\,\, u>0 & \text{in $\Omega$},\\
\,\, u=0 & \text{on $\partial\Omega$,}
\end{cases}
\end{equation*}
admits at least one solution, provided the ratio $R_2/R_1$ is sufficiently large \cite{coron_topology}. Subsequently, in a landmark paper \cite{bahri_coron_topology}, Bahri and Coron strengthened this result using topological methods, proving the existence of a solution under the condition that $H_m(\Omega, \mathbb{Z}_2) \neq 0$ for some $m > 0$, where $H_m$ denotes the homology group of dimension $m$ of $\Omega$ with $\mathbb{Z}_2$ coefficients. Further extensions to contractible domains were developed in \cite{dancer_topology, ding_topology, passaseo_topology}.

Extensions of the Coron problem have also been investigated for various nonlinear operators. The case of the $p$-Laplacian is treated in \cite{mercuri_coron_p_lap}, while the fractional Laplacian case is considered in \cite{secchi_coronpb_fractional}. For systems like the Lane–Emden system, analogous results are discussed in \cite{jin_coron_lane_emden}.

In light of the profile decomposition established in Theorem \ref{PS_decomposition} for the critical problem \eqref{Maineq}, we turn our attention to studying the corresponding Coron-type result. By constructing appropriate variational settings and employing concentration-compactness techniques, we obtain the following existence theorem for high energy solutions:

\begin{theorem}\label{thm:high_energy_soln}
Let $\Omega \subset \mathbb{R}^N$ be a bounded domain satisfying:
\begin{align}
(i)& ~\left\{x \in \mathbb{R}^N : R_1 < |x| < R_2 \right\} \subset \Omega,\label{eq:domain_property_1}\\
(ii)& ~\left\{x \in \mathbb{R}^N : |x| < R_1 \right\} \not\subset \overline{\Omega},\label{eq:domain_property_2}
\end{align}
for some $0 < R_1 < R_2 < +\infty$. Then, there exists $N_0 \in \mathbb{N}$ such that for all $N \geq N_0$, and $\frac{R_2}{R_1}$ sufficiently large depending on $N$, there exists a positive solution $u\in\X (\om)$ to \eqref{Maineq} with $\lambda = 0$, satisfying
$$
\frac{1}{N} \So^{N/2} < I_{0,s}(u) < \frac{2}{N} \So^{N/2}.
$$
\end{theorem}

The proof of Theorem \ref{thm:high_energy_soln} proceeds by contradiction. We consider a two-parameter family of functions and $R\gg 1$ defined as 

\begin{equation}\mathcal{A}_R\coloneqq \left\{ \phi_R U[t\sigma,(1-t)]\,:\, \phi_R \in C^{\infty}_c(\Rn),\,t \in [0,1) \mbox{ and } \sigma \in \mathbb{S}^{N-1} \right\},\label{familyOfBubbles:Intro}\end{equation}where $u_0$ denotes the standard Aubin-Talenti bubble defined in \eqref{eq:AubinTalentiBubbles}, $\phi_R$ is a suitable cut-off function, and the parameters $t$ and $\sigma$ control the concentration and translation of the bubbles within the domain. We then estimate the associated energy of this class and show that for sufficiently large dimension $N$ and an appropriately large ratio $\frac{R_2}{R_1}$, the energy remains strictly below the threshold $2^{2/N} \So$ (see Lemma \ref{lem:crucial_lemma}). This estimate plays a pivotal role in establishing the existence of a high-energy solution, as it allows us to deform the bubble family so that its energy approaches the first critical level. Assuming by contradiction that no positive nontrivial solution exists, this level $2^{2/N} \So$ becomes the next critical threshold after $\So$ at which the Palais–Smale condition breaks down.
Thus, applying a version of the deformation lemma, we construct a homotopy retracting the unit sphere $\mathbb{S}^{N-1} \subset \Omega$ continuously to a single point within $\Omega$. However, such a deformation contradicts the topological condition \eqref{eq:domain_property_2}. 
\smallskip
\begin{remark}
   \begin{enumerate}
    \item One of the natural limitations of our result lies in the assumption that the dimension $N$ must be sufficiently large in order to guarantee the existence of solutions. While this assumption is crucial for our analytical estimates to hold, numerical simulations performed using MATLAB suggest that Lemma~\ref{lem:crucial_lemma} may fail for lower dimensions. This raises an intriguing open question: to precisely determine the threshold dimension beyond which the energy associated with the localized family of Aubin-Talenti bubbles defined in \eqref{familyOfBubbles:Intro} falls strictly below the critical level $2^{2/N}\So$.
    
    \item An alternative approach to relax the large dimension assumption is to introduce a small parameter $\varepsilon = \varepsilon(N)$ in front of the fractional Laplacian operator. By tuning $\varepsilon$ appropriately, one can ensure that the energy lies below the second critical level $2^{2/N} \So$. This perturbative strategy could allow the extension of the existence result to lower-dimensional settings, opening up further avenues of investigation.

    \item Next, we aim to investigate the extent to which the geometric assumptions on $\Omega$ can be relaxed while still ensuring the existence of at least one positive solution to $\eqref{Maineq}$ with $\lambda = 0$.
\end{enumerate}
        
\end{remark}

\subsection{Notations} Throughout this paper, we make use of the following terminology unless, in a specific scenario, it's stated otherwise.

\begin{enumerate}
\item $\mathbb{N},\,\R$ denote the set of natural numbers and real numbers, respectively. Moreover we use $\mathbb{N}_0\coloneqq\mathbb{N}\cup\{0\}$ and $\R^+\coloneqq\{r\in\R\colon r>0\}=\(0,\infty\)$. For a generic set $B$ and $N\in\mathbb{N}$, the cartesian $N$-product is the set $B^N$.\vspace{1mm}

\item $B(x,r)$ is a ball of radius $r\in\R^+$ with center at $x \in \Rn$. $B(x,r)^c = \Rn \setminus B(x,r)$ and $B_R\coloneqq B(0,R)$.\vspace{1mm}

\item $\|\cdot\|_{L^p(U)}$ denotes the usual $L^p\(U\)$-norm. We define the weighted Lebesgue space with respect to a positive weight $w\in L^1_{\rm loc}\(\Rn\)$ as follows:
\be\no
L^p\(\Rn,w\) \coloneqq \left\{f\colon\Rn \to\R \,\colon\, f \text{ is measurable and}\int_{\Rn}|f|^{p}w \,{\rm d}x <\infty\right\}.
\ee The weighted norm $\|\cdot\|_{L^p\(\Rn,w\)}$ is given by $\|w^{\frac1p}\cdot\|_p=\left(\int_{\Rn} |\cdot|^p w\,{\rm d}x\right)^{\frac{1}{p}}$.\vspace{1mm}

\item $P\lesssim_{a}Q$ represents that there exists a constant $C\equiv C(a)>0$ such that $P\leq CQ$ holds. $P\approx_{a}Q$ means there exists $C(a),\,C'(a)>0$ such that $C'Q\leq P\leq CQ$. \vspace{1mm}

\item $\left(\X\right)'$ denotes the topological dual of the fractional homogeneous Sobolev space $\X (\Omega)$. We endow 
$\left(\X\right)'$ with the dual operator norm induced by $\X$.\vspace{1mm}

\item $C(a,b,\cdots)$ signifies a positive constant depending only on the parameters $a,b,\cdots$. When there is no confusion, other generic constants are used as $\tilde{C},\, C,\, C',\dots$ etc., and they may vary from line to line.\vspace{1mm}

\item $U[z,\la]$ denotes a general Aubin-Talenti bubble with translation parameter $z\in \Rn$ and dilation parameter $\la>0$. \vspace{1mm}

\item We regard $H_0^1(\bar{\Omega})$ as a subspace of $H^1(\mathbb{R}^N)$ via extension by zero outside $\bar{\Omega}$. 
\end{enumerate}

    \section{Preliminaries and Auxiliary lemmas}\label{sec:notations}
\noi    
We consider the variational framework associated with a combination of local and nonlocal operators. For a comprehensive treatment of such settings, we refer to \cite{biagi}, and the references therein.

\noindent
In this context, the following Sobolev-type inequality involving both local and fractional operators holds for functions $ u:\mathbb{R}^N \to \mathbb{R} $ that vanish outside an open set $\Omega \subset \mathbb{R}^N$:
\begin{equation}\label{eq:mixed_inequality}
\mathcal{S}_{N,s}(\Omega)\|u\|_{L^{2^*}(\mathbb{R}^N)}^2 \leq \|\nabla u\|_{L^2(\mathbb{R}^N)}^2 + \iint_{\mathbb{R}^{2N}}\frac{|u(x)-u(y)|^2}{|x-y|^{N+2s}}\,{\rm d}x\,{\rm d}y,    
\end{equation}
where $ N \geq 3 $, $ 2^*\coloneq \frac{2N}{N-2} $ is the critical Sobolev exponent, and the constant $ \mathcal{S}_{N,s}(\Omega) $ denotes the optimal (i.e., best possible) constant for which the inequality holds.
\noi
To this inequality, one can naturally associate the following Rayleigh-type quotient, 

\begin{equation*} 
   S(u; \Omega) := \frac{\|\nabla u\|_{L^2(\Omega)}^2 + [u]_s^2}{\|u\|_{L^{2^*}(\Omega)}^2}, \quad u \in \mathcal{C}_0^\infty(\Omega),
\end{equation*}
where $[u]_s$ stands for the Gagliardo seminorm, defined by
\begin{equation*}
[u]_s^2 := \iint_{\mathbb{R}^{2N}} \frac{|u(x)-u(y)|^2}{|x-y|^{N+2s}}\,{\rm d}x\,{\rm d}y.
\end{equation*}
Moreover, we will denote the inner product associated with the above seminorm by $\langle\cdot, \cdot\rangle_s$.

\subsection{Classical Sobolev inequality} \label{rem:propSnrecall}
Before delving further, let us outline some fundamental facts regarding the sharp constant $\So $ appearing in the classical Sobolev inequality. This constant is defined by
$$
\mathcal{S}_N := 
\inf\left\{
\frac{\|\nabla u\|^2_{L^2(\Omega)}}{\|u\|^2_{L^{2^*}(\Omega)}} : u \in C_0^\infty(\Omega ) \setminus \{ 0 \}
\right\},
$$
and is extensively discussed in the works \cite{aubin_sobolev_constant, talenti_sobolev_constant}. A notable feature of $\mathcal{S}_N$ is its independence from the domain $\Omega \subset \mathbb{R}^N$; its value depends only on the space dimension $N$. Moreover, an explicit expression for this constant is available and given by
\begin{equation*} 
\mathcal{S}_N = N(N-2)\pi \left( \frac{\Gamma(N/2)}{\Gamma(N)} \right)^{2/N} = \frac{N(N-2)}{4}\omega_N^{\frac{2}{N}}, 
\end{equation*}
where $ \Gamma(\cdot) $ denotes the Gamma function and $\omega_N\coloneq \frac{2\pi^{\frac{N+1}{2}}}{\Gamma(\frac{N+1}{2})}$ denotes the surface measure of $\mathbb{S}^N$ in $\R^{N+1}$.

It is also worth noting that when $ \Omega $ is a bounded domain, the infimum defining $ \mathcal{S}_N $ is never attained, i.e., no minimizer exists. However, in the particular case where $ \Omega = \mathbb{R}^N $, the situation is different: although the minimum is not achieved by any function in $ C_0^\infty(\mathbb{R}^N) $, it is attained in the larger space $ \mathcal{D}_0^{1,2}(\mathbb{R}^N) $. The family of extremal functions realizing the minimum in this setting is explicitly known and given by the so-called Aubin-Talenti bubbles \cite{aubin_sobolev_constant, talenti_sobolev_constant}. That is the set
\begin{equation}\label{Fam-Aub-Tal-Bub}
\mathcal{F} := \left\{
cU[z_0,\lambda](x) = c\lambda^{-\frac{(N-2)}{2}} U\left( \frac{x - z_0}{\lambda} \right) : \,c\in\R\setminus \{0\},\,\lambda > 0,\, z_0 \in \mathbb{R}^N\right\},
\end{equation}
where, 
\begin{equation*}
  U(z) := (1 + |z|^2)^{-\frac{(N-2)}{2}}.  
\end{equation*}

This family reflects the invariance properties of the problem under dilation and translation, and plays a central role in the analysis of Sobolev embeddings.

\medskip

\subsection{Function Space}\label{Function Space}

Let $\Omega \subseteq \mathbb{R}^N$ be a nonempty open subset, where $N \geq 3$, and not necessarily bounded. We define the function space $\X(\Omega)$ as the completion of $C_0^\infty(\Omega)$ with respect to the global norm
\begin{equation}\label{norm}
\rho(u) := \left( \|\nabla u\|_{L^2(\mathbb{R}^N)}^2 + [u]_s^2 \right)^{1/2}, \quad u \in C_0^\infty(\Omega),
\end{equation}
where $[u]_s$ denotes the fractional semi-norm. This norm is naturally associated with the scalar product
\begin{equation}\label{Inn-Prod}
    \langle u, v \rangle_\rho := \int_{\mathbb{R}^N} \nabla u \cdot \nabla v \, \,{\rm d}x + \langle u, v \rangle_s,
\end{equation}

where $\cdot$ represents the standard inner product in $\mathbb{R}^N$, and $\langle u, v \rangle_s$ corresponds to the bilinear form associated with the fractional part. Equipped with this inner product, the space $\mathcal{X}_0^{1,2}(\Omega)$ becomes a Hilbert space.
\medskip

In the case when $\Omega$ is bounded, it is useful to recall a fundamental inequality that expresses the continuous embedding of $H^1(\mathbb{R}^N)$ into $H^s(\mathbb{R}^N)$. Specifically, for every $u \in C_0^\infty(\Omega)$, there exists a constant $c =c(s) > 0$ such that
\begin{equation}
[u]_s^2 \leq C(s) \|u\|_{H^1(\mathbb{R}^N)}^2 = C(s)\left( \|u\|_{L^2(\mathbb{R}^N)}^2 + \|\nabla u\|_{L^2(\mathbb{R}^N)}^2 \right). \label{ContEmb}
\end{equation}

The fractional Sobolev space $H^s(\Omega)$, which lies between $L^2(\Omega)$ and $H^1(\Omega)$, is defined by
$$
H^s(\Omega) := \left\{ u \in L^2(\Omega) : \frac{u(x) - u(y)}{|x - y|^{\frac{N}{2} + s}} \in L^2(\Omega \times \Omega) \right\}.
$$
This space is equipped with the norm
$$
\|u\|_{H^s} := \left( \int_\Omega |u(x)|^2\, \,{\rm d}x + \int_\Omega \int_\Omega \frac{|u(x) - u(y)|^2}{|x - y|^{N + 2s}}\, \,{\rm d}x\, \,{\rm d}y \right)^{1/2}.
$$

The inequality \eqref{ContEmb}, together with the classical Poincaré inequality, ensures that the norm $\rho(\cdot)$ is equivalent to the standard $H^1(\mathbb{R}^N)$-norm on $C_0^\infty(\Omega)$. As a consequence, the space $\mathcal{X}_0^{1,2}(\Omega)$ coincides with the closure of $C_0^\infty(\Omega)$ in $H^1(\mathbb{R}^N)$, that is,
$$
\begin{aligned}
\mathcal{X}_0^{1,2}(\Omega) &= \overline{C_0^\infty(\Omega)}^{\|\cdot\|_{H^1(\mathbb{R}^N)}} \\
&= \left\{ u \in H^1(\mathbb{R}^N) : u|_\Omega \in H_0^1(\Omega) \text{ and } u \equiv 0 \text{ a.e. in } \mathbb{R}^N \setminus \Omega \right\}.
\end{aligned}
$$

We now conclude this section with a compactness result that will be crucial in the proof of the global compactness theorem, Theorem \ref{PS_decomposition}. To this end, we first recall a useful auxiliary lemma from \cite{giovanni}.

\begin{lemma}\label{LeoniIneq}
Let $\Omega \subseteq \Rn$ be an open set, $0<s_1<s_2<1$. Then $\forall u \in H^{s_2}(\Omega)$, there exists a constant $c>0$ such that 
\begin{equation*}
    \[u\]_{s_1}\leq c \|u\|_{L^2(\Omega)}^{1-s_1/s_2}\[u\]_{s_2}^{s_1/s_2}.
\end{equation*}
\end{lemma}

\begin{lemma}\label{compact_embedding}
Let $\Omega \subset \mathbb{R}^N$ be a bounded and nonempty open set. Then, the embedding $\mathcal{X}_0^{1,2}(\Omega) \hookrightarrow H^s(\Omega)$ is compact.
\end{lemma}
\begin{proof}
Let us consider a bounded sequence $\{u_k\}$ in the space $\X(\Omega)$. Since this sequence is bounded in $H^1(\Omega)$, it has a convergent subsequence (still denoted by $\{u_k\}$) in $L^2(\Omega)$. Consequently, $\{u_k\}$ forms a Cauchy sequence in $L^2(\Omega)$. 

    Furthermore, using the continuity of the embedding $H^1(\Omega) \hookrightarrow H^{s^\prime}(\Omega)$ (as established in \eqref{ContEmb} for some $s^\prime \in (s,1)$) and employing Lemma \ref{LeoniIneq}, we can show that $\{u_k\}$ is a Cauchy sequence in $H^s(\Omega)$. Indeed
    \begin{align*}
       [u_k-u_{m}]_s &\leq \|u_k-u_{m}\|_{L^2{(\Omega)}}^{1-s/s^{\prime}}[u_k-u_{m}]_{s^{\prime}}^{s/s^{\prime}}\\
       &\leq \|u_k-u_{m}\|_{L^2{(\Omega)}}^{1-s/s^{\prime}} \left([u_k]_{s^\prime}+[u_m]_{s^\prime}\right)^{s/s^{\prime}}\leq C  \|u_k-u_{m}\|_{L^2{(\Omega)}}^{1-s/s^{\prime}}.
    \end{align*}
The inequality above clearly implies that $\{u_k\}$ is a Cauchy sequence in $H^s(\Omega)$. Thus, the proof is complete.
\end{proof}
\medskip
\noindent
\noi
We next define the energy functional corresponding to the problem \eqref{Maineq}. For any $ u \in \X(\Omega)$, it is given by  
\begin{eqnarray}
I_{\lambda,s}(u) &:=& \frac{1}{2}\int_{\Rn} |\nabla u|^2\, \,{\rm d}x + \frac{1}{2} \iint_{\mathbb{R}^{2N}} \frac{|u(x)-u(y)|^2}{|x-y|^{N+2s}}\, \,{\rm d}x\, \,{\rm d}y - \frac{\lambda}{2} \int_{\Rn}|u|^2\, \,{\rm d}x  \no\\
& & - \frac{1}{2^*}\int_{\Rn} |u|^{2^*}\, \,{\rm d}x \\
&=& \frac{1}{2}\rho (u)^2 - \frac{\lambda}{2} \int_{\Omega} |u|^2\, \,{\rm d}x - \frac{1}{2^*}\int_{\Omega} |u|^{2^*}\, \,{\rm d}x.\label{Ener-Fun}
\end{eqnarray}
Moreover, the limiting functional associated with the problem at infinity becomes
\begin{equation*}
I_{\infty}(u) := \frac{1}{2}\int_{\Rn} |\nabla u|^2 \,{\rm d}x - \frac{1}{2^*}\int_{\Rn}|u|^{2^*}\,{\rm d}x.   
\end{equation*}

\section{Global Compactness Result}
 In this section, we prove the global compactness result for the problem \eqref{Maineq}, stated in Theorem \ref{PS_decomposition}.

\begin{proof}[Proof of Theorem \ref{PS_decomposition}] We divide the proof into several steps. \medskip

\noi
\textbf{\underline{Step~1}:} We first show that the sequence $\{u_k\}$ is bounded in $\X(\om)$. Consequently, up to a subsequence (still denoted by $\{u_k\}$), we may assume that $u_k \rightharpoonup u_0$ weakly in $\X(\om)$. Moreover, $u_0$ is a weak solution to the problem \eqref{Maineq}.

\medskip

\noi Let $\{u_k\}$ be a Palais–Smale sequence for $I_{\la,s}$ at level $\beta$, i.e., $I_{\la,s}(u_k) \to \beta$ and $I_{\la,s}'(u_k) \to 0$ in $(\X(\om))'$ as $k \to \infty$. Then, for some constant $C > 0$, we have
\begin{align*}
\beta +o(1)+o(1)\rho(u_k)
    \geq & \left|I_{\la,s}(u_k) - \frac{1}{2}\prescript{}{\DX}{\langle} I'_{\la,s}(u_k),\, u_k{\rangle}_{\X}\right| \\
    = & \left( \frac{1}{2} - \frac{1}{2^*} \right) \int_{\om}|u_k|^{2^*} \,{\rm d}x\\
    \geq & ~C \left( \int_{\om} |u_k|^2 \,{\rm d}x \right)^{\frac{2^*}{2}},
\end{align*}
for some constant $C > 0$. As a result, we obtain
\begin{align*}
    \rho(u_k)^2 &= 2 I_{\la,s}(u_k) + \la \int_{\om} |u_k|^2 \,{\rm d}x + \frac{2}{2^*} \int_{\om} |u_k|^{2^*} \,{\rm d}x \\
    &\leq 2\left(\beta+o(1)\right) + C\left(\beta+o(1)+o(1)\rho(u_k)\right)^{\tfrac{2}{2^*}} + C\left(\beta+o(1)+\rho(u_k)\right)\\
    &\leq C+o(1)\rho (u_k),
\end{align*}
for some constants $C>0$. This inequality shows that the sequence $\{u_k\}$ is bounded in $\X(\om)$. \\
Therefore, there exists $u_0 \in \X(\om)$ such that, up to a subsequence, $u_k \rightharpoonup u_0$ weakly in $\X(\om)$. This implies that for all $\phi \in \X(\om)$, 
\begin{align*}
\langle u_k,\phi\rangle_{\rho}:=&\int_{\om}\nabla u_k\cdot\nabla \phi\,{\rm d}x + \iint_{\R^{2N}}\frac{\left(u_k(x)-u_k(y)\right)\left(\phi (x)-\phi (y)\right)}{\abs{x-y}^{N+2s}}\,{\rm d}x\,{\rm d}y\no\\
&\rightarrow \langle u_0,\phi\rangle_{\rho}=\int_{\om}\nabla u_0\cdot\nabla \phi\,{\rm d}x + \iint_{\R^{2N}}\frac{(u_0(x)-u_0(y))(\phi (x)-\phi (y))}{\abs{x-y}^{N+2s}}\,{\rm d}x\,{\rm d}y,
\end{align*}
as $k \to \infty$.
\noi
Moreover, since $u_k \to u_0$ almost everywhere in $\om$, and using Vitali's convergence theorem, we obtain
$$
\int_{\om} \left( |u_k|^{2^* - 2} u_k - |u_0|^{2^* - 2} u_0 \right) \phi \, \,{\rm d}x \to 0, \quad \text{as } k \to \infty, \quad \text{for all } \phi \in \X(\om).
$$
\noi
In addition, since $\X(\om)$ is compactly embedded into $L^2(\om)$, it follows that
$$
\int_{\om} u_k \phi \, \,{\rm d}x \to \int_{\om} u_0 \phi \, \,{\rm d}x, \quad \text{as } k \to \infty, \quad \text{for all } \phi \in \X(\om).
$$
\noi
Putting these together, we conclude that $\prescript{}{\DX}{\langle} I_{\la,s}^{\prime}(u_0), \phi \rangle_{\X} = 0$ for all $\phi \in \X(\om)$. Hence, $u_0$ is a weak solution to the problem \eqref{Maineq}.

\bigskip
\noindent
   \textbf{\underline{Step~2}:} Define $v_k \coloneqq u_k - u_0$. We aim to show that the sequence $\{v_k\}$ forms a $(PS)$-sequence for the limiting functional $I_{\infty}$ at the level $\beta - I_{\la,s}(u_0)$.

\smallskip
 \noi
Since the embedding $\X(\om) \hookrightarrow H_0^1(\om)$ is continuous and $u_k \rightharpoonup u_0$ in $\X(\om)$, it follows that $v_k \rightharpoonup 0$ in $H_0^1(\om)$. This yields the identity
\begin{equation}
    \int_{\mathbb{R}^N} |\nabla v_k|^2 
    = \int_{\Omega} |\nabla v_k|^2 
    = \int_{\Omega} |\nabla u_k|^2 - \int_{\Omega} |\nabla u_0|^2 + o(1), 
    \label{BrezisLiebGrad}
\end{equation}
where $o(1)\rightarrow 0$ as $k \to \infty$.

\smallskip
\noi
Moreover, due to the compact embedding of $H_0^1(\Omega)$ into $L^2(\Omega)$ and by applying the Brezis–Lieb Lemma, we obtain
\begin{equation}
    o(1)=\int_{\Omega} |v_k|^2 = \int_{\Omega} |u_k|^2 - \int_{\Omega} |u_0|^2 + o(1).
\end{equation}

\smallskip
\noi
Similarly, the Brezis–Lieb Lemma applied to the critical term gives
\begin{equation}
    \int_{\Omega} |v_k|^{2^*} = \int_{\Omega} |u_k|^{2^*} - \int_{\Omega} |u_0|^{2^*} + o(1).
\end{equation}

\noi
\smallskip
In addition, employing Lemma \ref{compact_embedding} along with the Brezis–Lieb Lemma for the nonlocal term, we get
\begin{equation}
    o(1)= [v_k]_s^{2} = [u_k]_s^{2} - [u_0]_s^{2} + o(1).
    \label{BrezisLiebGagliardo}
\end{equation}

\medskip
\noindent
We now claim that $\{v_k\}$ is indeed a $(PS)$-sequence for $I_\infty$ at the level $\beta - I_{\la,s}(u_0)$. From \eqref{BrezisLiebGrad}-\eqref{BrezisLiebGagliardo}, we compute:
\begin{align*}
&I_{\infty}(v_k)\\
& =\frac{1}{2} \int_{\om} \abs{\nabla v_k}^2 -\frac{1}{2^*}\int_{\om}\abs{v_k}^{2^*}\,\no\\
& =\frac{1}{2}\int_{\om} \abs{\nabla u_k}^2\,  + \frac{1}{2}\iint_{\R^{2N}}\frac{\abs{u_k(x)-u_k(y)}^2}{|x-y|^{N+2s}}\,{\rm d}x\,{\rm d}y - \frac{\la}{2}\int_{\Omega} |u_k|^{2}\, -\frac{1}{2^*}\int_{\om}\abs{u_k}^{2^*}\, \no\\
&-\left\{\frac{1}{2}\int_{\om}\abs{\nabla u_0}^2\,  + \frac{1}{2}\iint_{\R^{2N}}\frac{\abs{u_0(x)-u_0(y)}^2}{|x-y|^{N+2s}}\,{\rm d}x\,{\rm d}y - \frac{\la}{2}\int_{\Omega} |u_0|^{2}\,  -\frac{1}{2^*}\int_{\om}\abs{u_0}^{2^*}\, \right\}+ o(1)\no\\
& =I_{\la,s}(u_k) - I_{\la,s}(u_0) +o(1)\rightarrow \beta - I_{\la,s}(u_0)\,\text{ as }k\to\infty.\no
\end{align*}

\medskip
\noindent
To verify the $(PS)$ condition, we consider
\begin{align*}
     &\prescript{}{\DX}{\langle} I'_{\infty}(u_k-u_0),\, \phi{\rangle}_{\X}\no\\
     &= \int_{\om} \nabla (u_k-u_0)\cdot\nabla \phi\,  - \int_{\om} \abs{u_k-u_0}^{2^*-2}(u_k-u_0)\phi\, \no\\
     &= \int_{\om}\nabla u_k\cdot \nabla \phi\, +\iint_{\R^{2N}}\frac{(u_k(x)-u_k(y))(\phi(x)-\phi(y))}{\abs{x-y}^{N+2s}}\,{\rm d}x\,{\rm d}y\no\\
     &-\la\int_{\om}u_k\phi\, -\int_{\om}\abs{u_k}^{2^*-2}u_k\phi\,  \no\\
     &-\left\{ \int_{\om}\nabla u_0\cdot \nabla \phi\, +\iint_{\R^{2N}}\frac{(u_0(x)-u_0(y))(\phi(x)-\phi(y))}{\abs{x-y}^{N+2s}}\,{\rm d}x\,{\rm d}y\right\}\\
     &+\la\int_{\om}u_0\phi\, +\int_{\om}\abs{u_0}^{2^*-2}u_0\phi\,  \no\\
     &+\la\int_{\om}(u_k-u_0)\phi\, -\int_{\om}\left\{\abs{u_k-u_0}^{2^*-2}(u_k-u_0)-\abs{u_k}^{2^*-2}u_k+\abs{u_0}^{2^*-2}u_0\right\}\phi\,  + o(1)\no\\
     &= \prescript{}{\DX}{\langle} I'_{\la,s}(u_k),\, \phi{\rangle}_{\X}- \prescript{}{\DX}{\langle} I'_{\la,s}(u_0),\, \phi{\rangle}_{\X}+o(1)=o(1),\no
\end{align*}
for all $\phi \in \X(\Omega)$, since $u_k$ is a P.S. sequence and $u_0$ is a weak solution. In fact we can say more. Using Lemma~\ref{compact_embedding}, we can easily see that $\{v_k\}$ is indeed a $(PS)$ sequence for $I_{\infty}$ at the level $\beta-I_{\la,s}$ on $H^1_0(\Omega)$.

\bigskip
\noindent
    \textbf{\underline{Step~3}:} Since $v_k=u_k - u_0 \rightharpoonup 0$ in $\X(\Omega)$, and consequently in $H_0^1(\Omega)$, applying Lemma~3.3 from \cite{Struwe-VM}, we deduce the following:
there exists a sequence of points $\{x_k\} \subseteq \Omega$, a sequence of radii $\{R_k\}$ with $R_k \to \infty$ as $k \to \infty$, a non-trivial solution $v^0$ to the limiting problem \eqref{eq:problem_at_infinity}, and a P.S. sequence $\{w_k\}$ for $I_\infty$ in $H_0^1(\Omega)$ such that for a subsequence $\{v_k\}$, the following holds:
$$
w_{k}=v_{k}-R_{k}^{\frac{N-2}{2}} v^{0}\left(R_{k}\left(\cdot-x_{k}\right)\right)+o(1),
$$
where $o(1) \to 0$ in $\mathcal{D}^{1,2}\left(\mathbb{R}^N\right)$ as $k \to \infty$. Notably, $w_k \rightharpoonup 0$ weakly. Additionally, the energy satisfies
\be\no
I_{\infty}\left(w_{k}\right)=I_{\infty}\left(v_{k}\right)-I_{\infty}\left(v^{0}\right)+o(1) .
\ee

\medskip
\noindent
    \textbf{\underline{Step~4}:} From step $2$, we have
$$
I_{\lambda,s}\left(u_{k}\right)=I_{\lambda,s}\left(u_0\right)+I_{\infty}\left(v_{k}\right)+o(1).
$$
 \noi   
We proceed to conclude the proof in this step by iterating the process outlined in Step 3 as follows: Step 3 is applied to the sequences defined by $v_k^1 = u_k - u_0$, and for $j > 1, v_{k}^{j}=u_{k}-u_{0}-\sum_{i=1}^{j-1} U_{k}^{i}=v_{k}^{j-1}-U_{k}^{j-1}$, where

\be\no
U_{k}^{i}(x)=\left(R_{k}^{i}\right)^{\frac{N-2}{2}} U_{i}\left(R_{k}^{i}\left(x-x_{k}^{i}\right)\right),
\ee

\smallskip
\noi
where $U_i$ satisfies the equation 
\eqref{eq:problem_at_infinity}. Now using induction, we establish:
$$
\begin{aligned}
I_{\infty}\left(v_{k}^{j}\right) & =I_{\la,s}\left(u_{k}\right)-I_{\la,s}\left(u_{0}\right)-\sum_{i=1}^{j-1} I_{\infty}\left(U_{i}\right) \\
& \leq I_{\la,s}\left(u_{k}\right)-(j-1) \beta^{*},
\end{aligned}
$$
\noi
where $\beta^* = \frac{1}{N}\mathcal{S}_N^{N/2}$.

\noi
Since the last term becomes negative for sufficiently large $j$, the induction terminates after some index $l \geq 0$. Moreover, for this terminal index, we have

\be\no
v_{k}^{l+1}=u_{k}-u_{0}-\sum_{j=1}^{l} U_{k}^{j} \rightarrow 0
\ee
\noi
strongly in $\mathcal{D}^{1,2}\left(\mathbb{R}^N\right)$. Moreover,

\be\no
I_{\la,s}\left(u_{k}\right)-I_{\la,s}\left(u_{0}\right)-\sum_{j=1}^{l} I_{\infty}\left(U_{j}\right) \rightarrow 0.
\ee

\noi
Finally, the interaction of the bubbles follows directly from \cite{bahri_coron_topology}.
    
\end{proof}

\medskip

\medskip
\section{Positive solutions on annular-shaped domains: an application of profile decomposition}

In this section, we construct a high-energy positive solution to the critical exponent problem involving the mixed local–nonlocal operator on $\Omega$, with $\lambda = 0$ in equation~\eqref{Maineq}, leveraging the topological properties of $\Omega$. The argument follows in the spirit of Coron \cite{coron_topology} (see also \cite[Theorem~3.4]{Struwe-VM}).

To establish Theorem~\ref{thm:high_energy_soln}, we establish some preliminary lemmas and notational conventions.

\medskip

Our strategy involves constructing a two-parameter family of Aubin–Talenti–type bubbles whose energy lies strictly between the first and second energy levels. Assume that the domain $\Omega$ satisfies conditions~\eqref{eq:domain_property_1} and~\eqref{eq:domain_property_2} of Theorem~\ref{thm:high_energy_soln}. For simplicity of exposition, we assume $R_{1}=(4 R)^{-1}<1<4 R=R_{2}$ and $R\gg 1$. 
\noi
For $\sigma \in \mathbb{S}^{N-1} :=\left\{x \in \mathbb{R}^{N} ;|x|=1\right\}, x \in \mathbb{R}^{N}$ and $ t \in [0,1)$, consider the following $2$-parameter family of functions
$$
u_{t}^{\sigma}(x)\coloneqq U[t\sigma, (1-t)](x) =\left[\frac{1-t}{(1-t)^{2}+|x-t \sigma|^{2}}\right]^{\frac{N-2}{2}} \in \mathcal{D}^{1,2} \left(\mathbb{R}^{N}\right).
$$
\noi These functions are motivated from the fact that $\So$ is attained at each of $u_{t}^{\sigma}$ and $\So = \mathcal{S}_{N,s}$ as discussed in Section~\ref{sec:notations}.
Now, for any $\sigma \in \mathbb{S}^{N-1}$, we have
\begin{equation}
u_{t}^{\sigma} \rightarrow U[0,1]=u_{0}=\left[\frac{1}{1+|x|^{2}}\right]^{\frac{N-2}{2}} \quad \text{as } t \searrow 0. \label{eq:AubinTalentiBubbles}
\end{equation}
\noi Moreover, $u_{t}^{\sigma}$ \textit{concentrates} at $\sigma$ as $t \nearrow 1$. 
To localize these functions within the domain $\Omega$, consider a radially symmetric cut-off function $\varphi \in C_0^\infty(\mathbb{R}^N)$ satisfying
$$\varphi \equiv 1  \text{ on } \left\{x : \frac{1}{2}<|x|<2\right\}, \quad \varphi \equiv 0 \text{ outside } \left\{x : \frac{1}{4}<|x|<4\right\},\quad 0 \leq \varphi \leq 1 \text{ on } \Omega.$$
\noi
For each $R \geq 1$, define the rescaled cut-off
$$
\varphi_{R}(x)= \begin{cases}\varphi(R x), & 0 \leq|x|<R^{-1} \\ 1, & R^{-1} \leq|x|<R \\ \varphi(x / R), & R \leq|x|.\end{cases}
$$

\noi
Now with the appropriate cut-off function in hand, we define 
$$
w_{t,R}^\sigma=u_{t}^{\sigma} \cdot \varphi_{R},\;\; w_{0,R}=u_{0} \cdot \varphi_{R} \in \X(\Omega).
$$
\noi

\begin{lemma}\label{Lem:seminormConvergence}
With the notations introduced above, the following convergence holds $$\left\|w_{t,R}^\sigma-u_{t}^{\sigma}\right\|_{\X(\Rn)} \rightarrow 0 \text{ as } R \to \infty,$$ uniformly with respect to $\sigma \in \mathbb{S}^{N-1}$ and $t \in [0,1)$.
\end{lemma}
\begin{proof}

Following the estimates established in~\cite[Theorem~3.4]{Struwe-VM}, we obtain
\begin{align}
\int_{\mathbb{R}^{N}}\left|\nabla\left(w_{t,R}^\sigma-u_{t}^{\sigma}\right)\right|^{2} {\rm d}x &\leq C \int_{\left(\mathbb{R}^{N} \backslash B_{2 R}\right) \cup B_{(2 R)^{-1}}}\left|\nabla u_{t}^{\sigma}\right|^{2} {\rm d}x \notag\\
&+C \cdot R^{-2} \int_{B_{4 R} \backslash B_{2 R}}\left|u_{t}^{\sigma}\right|^{2} {\rm d}x+ CR^{2} \int_{B_{(2 R)^{-1}}}\left|u_{t}^{\sigma}\right|^{2} {\rm d}x \rightarrow 0 \label{gradEstimate}
\end{align}
as $R \rightarrow \infty$, uniformly in $\sigma \in \mathbb{S}^{N-1}\;\text{and }t \in[0,1)$. 

\smallskip
    Now, we will be done once we show the convergence of $ w_{t,R}^\sigma $ to $ u_{t}^\sigma $ in the Gagliardo seminorm uniformly for $ t \in [0,1) $ and $ \sigma \in \mathbb{S}^{N-1}$, i.e., we would like to show that
$
\sup_{t ,\sigma} \left[ w_{t,R}^{\sigma} - u_t^\sigma \right]_s \to 0 \text{ as } R \to \infty.
$
Indeed, by \cite[Theorem~6.21]{giovanni}, it suffices to prove that
$
\|w_{t,R}^\sigma - u_{t}^{\sigma}\|_{L^2(\mathbb{R}^N)} \to 0.
$

\noi
We compute
\begin{align*}
    \int_{\Rn}|w_{t,R}^\sigma-u_{t}^{\sigma}|^2 {\rm d}x &= \int_{\Rn}(1-\varphi_R(x))^2|u_{t}^{\sigma}|^2 \,{\rm d}x\\
    &\leq \int_{\left(\mathbb{R}^{N} \backslash B_{2 R}\right) \cup B_{(2 R)^{-1}}}|u_{t}^{\sigma}|^2 {\rm d}x=\underbrace{\int_{\left(\mathbb{R}^{N} \backslash B_{2 R}\right)}|u_{t}^{\sigma}|^2\, {\rm d}x}_{\mathcal{I}_1}+\underbrace{\int_{ B_{(2 R)^{-1}}}|u_{t}^{\sigma}|^2 {\rm d}x}_{\mathcal{I}_2}.  
\end{align*}
\noi  
We estimate each term separately.\\
\noindent
\textbf{Estimate for $ \mathcal{I}_1 $:}
\begin{align*}
    \mathcal{I}_1:= \int_{\left(\mathbb{R}^{N} \backslash B_{2 R}\right)}|u_{t}^{\sigma}|^2 {\rm d}x&= \int_{\mathbb{R}^{N} \backslash B_{2 R}}\left[\frac{1-t}{(1-t)^{2}+|x-t \sigma|^{2}}\right]^{N-2}\,{\rm d}x\\
    &=\frac{1}{(1-t)^{N-2}}\int_{\mathbb{R}^{N} \backslash B_{2 R}}\frac{1}{\left(1+\frac{|x-t\sigma|^2}{(1-t)^2}\right)^{N-2}}\,{\rm d}x\\
    &=(1-t)^2\int_{\Rn\backslash \tilde{B}}\frac{1}{(1+|y|^2)^{N-2}}\,{\rm d}y,
\end{align*}
where in the last step, we have performed the following change of variables:
$$\text{Let } y=\frac{x-t\sigma}{1-t} \text{ and } \tilde{B}= \left(\frac{B_{2R}-t\sigma}{1-t}\right).$$
Now note that 
$y\notin \tilde{B}\implies (1-t)y+t\sigma \notin B_{2R} \implies |y|\geq \frac{2R-t}{1-t} = \frac{2R-1}{(1-t)}+1\geq 2R.$\\Thus
\begin{align*}
    \mathcal{I}_1\;\leq \;\int_{\Rn\backslash B_{2R}}\frac{1}{(1+|y|^2)^{N-2}}\,{\rm d}y.
\end{align*}
Therefore, $\mathcal{I}_1\rightarrow0$ as $R \rightarrow \infty$ independent of $t$ and $\sigma$ for $N \geq 5.$\\

\noi
\textbf{Estimate for $ \mathcal{I}_2 $:}
\noi H\"older's inequality, together with the $ L^{2^*}(\mathbb{R}^N)$-norm invariance of the Aubin--Talenti bubbles, yields
\be\no
\sup_{t,\sigma} \|u_t^{\sigma}\|_{L^2\left(B_{(2R)^{-1}}\right)} \leq \|u_0\|_{L^{2^*}(\Rn)} \left|B\left(0,(2R)^{-1}\right)\right|^{\tfrac{1}{N}}\to 0 \text{ as }R\to\infty.
\ee

 \noi  
Combining all estimates, we conclude that
$
\|w_{t,R}^\sigma - u_t^\sigma\|_{L^2(\mathbb{R}^N)} \to 0 \;\text{ as }\; R \to \infty,
$
uniformly with respect to $ t \in [0,1) $ and $ \sigma \in \mathbb{S}^{N-1} $.
Therefore, invoking the preceding lemma together with \eqref{gradEstimate}, we obtain
$$ w_{t,R}^\sigma \rightarrow u_t^\sigma \;\text{ in }\; \X(\rnn) \text{ as } R\to\infty.$$

\end{proof}

\begin{lemma}\label{lemma:less_than_second_level}
Let us consider the normalized functions
$$
v_{t,R}^{\sigma}=w_{t,R}^\sigma /\left\|w_{t,R}^\sigma\right\|_{L^{2^{*}}(\rnn)}, v_{0,R}=w_{0,R} /\left\|w_{0,R}\right\|_{L^{2^{*}}(\rnn)}.
$$
Then, for sufficiently large $R \gg 1$, there exists a positive constant $S_{1} \in \mathbb{R}$, $N_0 \in \mathbb{N}$ such that
\begin{equation}\label{eq:less_than_second_level}
\sup _{(t,\sigma) \in \[0,1\)\times\mathbb{S}^{N-1}} S\left(v_{t,R}^{\sigma} ; \Omega\right)<S_{1}<2^{2 / N} \So
\end{equation}
\noi
for  all $N\geq N_0$.
\end{lemma}
To establish Lemma~\ref{lemma:less_than_second_level}, we prove the following crucial estimate, which plays a key role in the analysis of the Coron-type problem addressed in this section.

\begin{lemma}\label{lem:crucial_lemma}
Let $u_0$ be as defined in \eqref{eq:AubinTalentiBubbles}. Then for sufficiently large $N$, the following inequality holds
$$
\frac{\|\nabla u_0\|^2_{L^2(\mathbb{R}^N)} + (1-t)^{2-2s} [u_0]^2_{s}}{\|u_0\|^2_{L^{2^*}(\mathbb{R}^N)}} < 2^{2/N} \So.
$$
\end{lemma}
\begin{proof}
It is enough to show that
\begin{equation}
    [u_0]^2_s < \left(2^{2/N} - 1\right) \int_{\mathbb{R}^N} |\nabla u_0|^2 {\rm d}x\label{UpperBoundGaglNorm}
\end{equation}
Firstly, we analyze the left-hand side.

\noi
\medskip
The fractional norm is given by
\begin{equation*}
    [u_0]^2_s = \iint_{\mathbb{R}^{2N}} \frac{|u_0(x) - u_0(y)|^2}{|x-y|^{N+2s}} \,{\rm d}y {\rm d}x.
\end{equation*}
Applying standard integral estimates as obtained in \cite[Theorem 6.21]{giovanni}, we get
\begin{equation*}
    [u_0]^2_s \leq \frac{\omega_{N-1}}{s \ell^{2s}} \int_{\mathbb{R}^N} |u_0(x)|^2 {\rm d}x + \frac{\omega_{N-1} \ell^{2(1-s)}}{2(1-s)} \int_{\mathbb{R}^N} |\nabla u_0(x)|^2 {\rm d}x,
\end{equation*}
where $\omega_N = \frac{2 \pi^{\frac{N+1}{2}} }{\Gamma \left( \frac{N+1}{2}\right)},\;\ell>0$.
Rearranging the terms,
\begin{equation*}
  [u_0]^2_s \leq\frac{\omega_{N-1}}{2} \left( \frac{2}{s \ell^{2s}} \int_{\mathbb{R}^N} |u_0(x)|^2 {\rm d}x + \frac{\ell^{2(1-s)}}{1-s} \int_{\mathbb{R}^N} |\nabla u_0(x)|^2 {\rm d}x \right).
\end{equation*}

\noi
\medskip
Now, consider the following notations
\begin{equation*}
    A = \frac{2}{s} \int_{\mathbb{R}^N} |u_0(x)|^2 {\rm d}x, \quad B = \frac{1}{1-s} \int_{\mathbb{R}^N} |\nabla u_0(x)|^2 {\rm d}x,
\end{equation*}
\begin{equation*}
    a = 2s, \quad b = 2(1-s),
\end{equation*}
we obtain
\begin{equation*}
    [u_0]^2_s \leq \frac{\omega_{N-1}}{2} g(\ell),
\end{equation*}
where
\begin{equation*}
    g(\ell) = A \ell^{-a} + B \ell^b,\;\; \ell>0.
\end{equation*}
Taking the infimum over all $ \ell > 0 $, we get
\begin{align*}
    [u_0]^2_s& \leq \frac{\omega_{N-1}}{2}\operatorname{inf}_{l>0}g(l)=\frac{\omega_{N-1}}{2} \left( \left(\frac{b}{a}\right)^{\frac{a}{a+b}} + \left(\frac{a}{b}\right)^{\frac{b}{a+b}} \right) A^{\frac{b}{a+b}} B^{\frac{a}{a+b}}.\\
    &= \frac{\omega_{N-1}}{2} \left( \left(\frac{2(1-s)}{2s}\right)^{s} + \left(\frac{2s}{2(1-s)}\right)^{1-s} \right) \left(\frac{2}{s} \int_{\mathbb{R}^N} |u_0(x)|^2 {\rm d}x\right)^{1-s} \\
   & \,\,\times \left(\frac{1}{1-s} \int_{\mathbb{R}^N} |\nabla u_0(x)|^2 {\rm d}x\right)^s.
\end{align*}
Thus,

\begin{equation*}
    [u_0]_s^2 \leq \frac{\omega_{N-1} 2^{-s}}{s (1 - s)} \left( \int_{\mathbb{R}^N} |u_0(x)|^2  \,{\rm d}x \right)^{1 - s} \left( \int_{\mathbb{R}^N} |\nabla u_0(x)|^2  \,{\rm d}x\right)^s.
\end{equation*}
\noi
Now, the inequality \eqref{UpperBoundGaglNorm} holds if the following condition is met:

\begin{equation*}
    \frac{\omega_{N-1} 2^{-s}}{s (1 - s)} \left( \int_{\mathbb{R}^N} |u_0(x)|^2  \,{\rm d}x\right)^{1 - s} \left( \int_{\mathbb{R}^N} |\nabla u_0(x)|^2  \,{\rm d}x\right)^s \leq \left( 2^{\frac{2}{N}} - 1 \right) \left( \int_{\mathbb{R}^N} |\nabla u_0|^2  \,{\rm d}x\right).
\end{equation*}
\noi
Rearranging the terms,

\begin{equation*}
    \left[\frac{\omega_{N-1} 2^{-s}}{s (1 - s) \left( 2^{\frac{2}{N}} - 1 \right)} \right]^{\frac{1}{1-s}}\leq \frac{\int_{\mathbb{R}^N} |\nabla u_0|^2  \,{\rm d}x}{\int_{\mathbb{R}^N} |u_0|^2 \,{\rm d}x}. 
\end{equation*}

\noi
Now, evaluating the terms on the right-hand side of the above computation, we get

\begin{equation*}
    \int_{\mathbb{R}^N} |u_0(x)|^2 \,{\rm d}x = \omega_{N-1} \int_0^{\infty} \frac{r^{N-1}}{(1 + r^2)^{N-2}} \, dr
= \omega_{N-1} \frac{\Gamma\left(\frac{N}{2} -2\right) \Gamma\left( \frac{N}{2} \right)}{2 \Gamma(N-2)}.
\end{equation*}
\noi
Similarly, evaluating the integral of the gradient term,

\begin{equation*}
    \int_{\mathbb{R}^N} |\nabla u_0(x)|^2 \,{\rm d}x = \So \left( \int_{\mathbb{R}^N} |u_0(x)|^{2^*} \,{\rm d}x \right)^{2/2^*}.
\end{equation*}
\noi
Performing the integral computation,

\begin{equation*}
    \int_{\mathbb{R}^N} |\nabla u_0(x)|^2 \,{\rm d}x = \So \left( \omega_{N-1} \int_0^{\infty} \frac{r^{N-1} \,dr}{(1 + r^2)^N} \right)^{(N-2)/N}.
\end{equation*}
\noi
Simplifying further,

\begin{equation*}
    \int_{\mathbb{R}^N} |\nabla u_0(x)|^2 \,{\rm d}x = \frac{\So^2}{N (N - 2) \pi} \left( \frac{\Gamma(N)}{\Gamma(N/2)} \right)^{2/N} \left( \omega_{N-1} \frac{2^{-N} \sqrt{\pi} \Gamma(N/2)}{\Gamma\left( \frac{1+N}{2} \right)} \right)^{(N-2)/N}.
\end{equation*}
We have
\begin{align} \label{eq:bound}
\left[\frac{\omega_{N-1} 2^{-s}}{s (1 - s) \left( 2^{\frac{2}{N}} - 1 \right)} \right]^{\frac{1}{1-s}} \leq \So^2 \;\underbrace{2^{3-\frac{2}{N} - N} \pi^{-\frac{3}{2} - \frac{1}{N}} \frac{\Gamma(-2+N) \left( \frac{\Gamma(N)}{\Gamma(N/2)} \right)^{2/N} \Gamma\left(\frac{1+N}{2}\right)^{-1 + 2/N}}{N(N-2) \Gamma(-2 + \frac{N}{2})}}_{:=R(N)}.
\end{align}
\noi
We claim that for inequality \eqref{eq:bound}
\begin{equation*}
    \lim_{N\to \infty} \frac{\text{LHS}}{\text{RHS}} = 0.
\end{equation*}
\noi
Using the Taylor expansion, we note that
\begin{equation*}
    2^{2/N} = 1 + \frac{2}{N} \log(2) + O\left(\frac{1}{N^2}\right).
\end{equation*}
\noi
Further, recall the Stirling approximation
\begin{equation}
    \lim_{N \to \infty} \frac{\Gamma(N+1)}{\sqrt{2\pi N} \left(\frac{N}{e}\right)^N} = 1.
    \label{eq:stirling}
\end{equation}
\noi
Using the Stirling approximation \eqref{eq:stirling}, we obtain
\begin{align*}
    \lim_{N \to \infty} \frac{\omega_{N-1}}{2^{2/N}-1} &= \lim_{N \to \infty} \frac{2\pi^{N/2}}{\Gamma\left(\frac{N}{2}\right) (2^{2/N}-1)}\\
    &=
    \lim_{N \to \infty} \frac{2\pi^{N/2}}{\left(\frac{1}{N} \left[ 2 \log 2 + O\left(\frac{1}{N}\right) \right] \sqrt{2\pi (N-2)/2} \left(\frac{N-2}{2e}\right)^{(N-2)/2} \right)}\\
      & =\lim_{N \to \infty} \frac{ \pi^{N/2}  N}{\sqrt{2\pi} \log 2  \sqrt{\frac{N-2}{2}} \left(\frac{N-2}{2e}\right)^{(N-2)/2}}=0.
\end{align*}
 \smallskip
 \noi
Next, we recall the duplication formula for the Gamma function:
\begin{equation*}
    \Gamma(N) \Gamma\left(N+\frac{1}{2}\right) = 2^{1-2N} \sqrt{\pi} \, \Gamma(2N).
\end{equation*}
Using this identity, $R(N)$ can be simplified to
\begin{equation*}
    \pi^{-2} 2^2 \frac{(N-4)}{N-2} \frac{1}{2^{4/N}} \frac{1}{\pi^{2/N}} \frac{\left[ \Gamma\left(\frac{N+1}{2}\right) \right]^{4/N}}{N(N-1)}.
\end{equation*}
\noi
Moreover, by applying Stirling’s approximation, one readily obtains
\begin{equation*}
    \lim_{N \to \infty} R(N) = \frac{1}{\pi^2 e^2}.
\end{equation*}
\noindent Furthermore, a similar application of Stirling’s approximation yields
$$\lim_{N\rightarrow \infty}\So=\infty.$$
Thus the claim follows.
\end{proof}
\begin{proof}[Proof of Lemma \ref{lemma:less_than_second_level}] Since $$\left\|w_{t,R}^\sigma-u_{t}^{\sigma}\right\|_{\mathcal{D}^{1,2}(\rnn)} \rightarrow 0 \text{ and }[w_{t,R}^\sigma-u_{t}^{\sigma}]_s \rightarrow 0 \text{ as } R \rightarrow \infty,$$  by Lemma~\ref{Lem:seminormConvergence}, and using the identities
$$\left\|u_{t}^{\sigma}\right\|_{\mathcal{D}^{1,2}(\rnn)}=\left\|u_{0}\right\|_{\mathcal{D}^{1,2}(\rnn)} \text{ and } [u_t^\sigma]_s=(1-t)^{1-s}[u_0]_s, $$we deduce that
$$
S\left(v_{t,R}^{\sigma} ; \Omega\right)=S\left(w_{t,R}^\sigma ; \Omega\right) \rightarrow S\left(u_{t}^{\sigma} ; \mathbb{R}^{N}\right)=\frac{\lVert \grad u_0 \rVert_{L^2(\rnn)}^2 + (1-t)^{2-2s}[u_0]_s^2}{\lVert u_0 \rVert_{L^{2^*}(\rnn)}^2}
$$
\noi
as $R \rightarrow \infty$, uniformly in $\sigma \in \mathbb{S}^{N-1}, 0 \leq t<1$. 
\noi
Moreover, Lemma~\ref{lem:crucial_lemma} ensures the existence of an integer $N_0 \equiv N_0(s)\in \mathbb{N}$ such that for all $ N \geq N_0$, the following strict inequality holds:
$$\frac{\lVert \grad u_0 \rVert_{L^2(\rnn)}^2 + (1-t)^{2-2s}[u_0]_s^2}{\lVert u_0 \rVert_{L^{2^*}(\rnn)}^2}< 2^{\frac{2}{N}}\So, \;\; \forall N \geq  N_0.$$
This concludes the proof.
\end{proof}
Having established the key preparatory results, we are now in a position to complete the proof of Theorem~\ref{thm:high_energy_soln}. We proceed by applying a deformation argument to lower the energy in a neighbourhood of the first level and subsequently construct a contraction of the unit sphere $ \mathbb{S}^{N-1} $ into $\overline{\Omega} $ using the two-parameter family introduced earlier.

\begin{proof}[Proof of Theorem \ref{thm:high_energy_soln}]
Let us fix a value of $R$ and assume, for the sake of contradiction, that the problem \eqref{Maineq} with $\lambda =0$ does not admit a positive solution. By Theorem \ref{PS_decomposition}, the functional $ I_{0,s} $ satisfies the Palais-Smale condition $(PS)_{\beta}$ on the space $ \X(\Omega) $ for $ \frac{1}{N} \So^{N / 2} < \beta < \frac{2}{N} \So^{N / 2} $. Consequently, $ S(\cdot ; \Omega) $ satisfies the Palais-Smale condition $(PS)_{\beta}$ on $\mathcal{M}$ for $ \So < \beta < 2^{2/N} \So $, where $\mathcal{M}$ denotes the $L^{2^*}(\Omega)$ unit sphere in $\X(\Omega)$, defined as
\begin{equation*}
    \mathcal{M}=\left\{u \in \X(\Omega) : \int_{\Omega}|u|^{2^{*}} d x=1\right\}.
\end{equation*}

Furthermore, $ S(\cdot; \Omega) $ does not admit a critical value within this range.

\smallskip
\noi Now, for any $ \beta \in (\So, 2^{2/N} \So) $, by the deformation lemma \cite[Theorem II.3.11]{Struwe-VM}, there exists $ \varepsilon > 0 $ and a flow $ \Phi_{\beta}: \mathcal{M} \times [0,1] \to \mathcal{M} $ such that
$$
\Phi_{\beta}\left(M_{\beta+\varepsilon}, 1\right) \subset M_{\beta-\varepsilon},
$$
\noi where $M_{\beta}=\{u \in \mathcal{M} : S(u; \Omega)<\beta\}$.

\smallskip
\noindent
For any $ \delta > 0 $, by compactness, we can cover the interval $ [\So + \delta, S_1] $ by finitely many such $ \varepsilon $-intervals and compose the corresponding deformations $ \Phi_{\beta} $ to construct a flow $ \Phi: \mathcal{M} \times [0, 1] \to \mathcal{M} $ such that
$$
\Phi\left(M_{S_1}, 1\right) \subset M_{\So + \delta},
$$
where $ S_1 $ is as defined in \eqref{eq:less_than_second_level}. Additionally, we may assume that during the deformation for all $t\in [0,1]$ we have $ \Phi(u, t) = u $ for all $ u $ with $ S(u; \Omega) \leq \So + \frac{\delta}{2}$.

\smallskip
\noindent Consider the \emph{center of mass} map $F : \mathcal{M} \to \mathbb{R}^N$ defined as
\begin{equation*}
F(u)= \frac{\int_{\Omega} x|\nabla u|^{2} {\rm d}x}{\int_{\om} |\grad u|^2{\rm d}x}.
\end{equation*}
\noi
\noindent We claim that for any open neighborhood $U$ of $\bar{\Omega}$, there exists a $\delta > 0$ such that
$$ F(M_{\So + \de}) \subset U.$$ \noi 
Suppose, on the contrary, that no such $\delta > 0$ exists. Then, for every $k \in \mathbb{N}$, we have $F(M_{\So + \frac{1}{k}}) \not\subset U$. Consequently, there exists a sequence $\{ u_k \}$ with $u_k \in M_{\So + \frac{1}{k}}$ such that $F(u_k) \notin U$. \\
Since $S(u_k; \Omega) < \So + \frac{1}{k}$ and by definition $S(u_k; \Omega) \geq \So$, we conclude
$$
\lim_{k \to \infty} S(u_k; \Omega) = \So,
S(u_k;\om) \geq \So.$$ This gives 
$$\displaystyle \lim_{k \to \infty} S(u_k;\om) = \So,$$
i.e., $\{u_k\}$ is a minimizing sequence for $\So$.

\smallskip
\noindent By Ekeland's Variational Principle \cite[Theorem 8.5]{willem2012minimax}, there exists a Palais-Smale sequence $\{v_k\}$ for $S(\cdot; \Omega)$ at the level $\So$. By Step 1 of Theorem \ref{PS_decomposition}, we have that $\{v_k\}$ is bounded in $\X(\Omega)$, and up to a subsequence, $v_k \rightharpoonup v$ weakly in $\X(\Omega)$. Setting $w_k := \So^{1/(2^*-2)} v_k$, we obtain a Palais-Smale sequence for the functional $I_{0,s}$ at the level $\frac{1}{N} \So^{N/2}$. Once again, Theorem \ref{PS_decomposition} ensures $I_{0,s}\left(\So^{1/(2^*-2)} v\right) \leq \frac{1}{N} \So^{N/2}.$ However, Lemma~\ref{signChangSolnEnergy} asserts that the energy of any sign-changing critical point of $I_{0,s}$ is at least $\frac{2}{N} \So^{N/2}$. Hence, we must have $v \geq 0$. By applying Theorem~8.4 of \cite{garain2022regularity}, and using the equivalence of critical points of $I_{0,s}$ and $S(\cdot; \Omega)$, we conclude that either $v \equiv 0$ or $v > 0$. Thus our assumption enforces $v \equiv 0$.
\smallskip
\noindent Moreover, since $\|v_k\|_{L^{2^*}(\Omega)} = 1$ and
$$
\int_{\Omega} |\nabla v_k|^2\,{\rm d}x + [v_k]_s^2 \to \So \quad \text{as } k \to \infty,
$$
we infer that $\|\nabla v_k\|_{L^2(\Omega)}$ is bounded. By Lemma \ref{compact_embedding}, it follows that $[v_k]_s^2 \to 0$ as $k \to \infty$, and thus
$$
\|\nabla v_k\|^2_{L^2(\Omega)} \to \So.
$$

\smallskip
\noindent Additionally, we consider the weak convergence of measures:
$$
|\nabla v_k|^2 \, dx \rightharpoonup \mu, \quad |v_k|^{2^*} \, dx \rightharpoonup \nu,
$$
as $k \to \infty$ in the space of Radon measures.
Then, using the identity $\So^{-1} \|\mu\| = \|\nu\|^{2/2^*}$ and Lemma 1.40 in \cite{willem2012minimax}, we deduce that both $\mu$ and $\nu$ are concentrated at a single point, say $x^{(0)}$. Hence, up to a subsequence,
\begin{equation*}
|v_k|^{2^*} \, {\rm d}x \rightharpoonup \delta_{x^{(0)}}, \;\;
|\nabla v_k|^2 \, {\rm d}x \rightharpoonup \So \delta_{x^{(0)}}.
\end{equation*}

\smallskip
\noindent It follows that
$$
F(v_k) = \frac{\int_{\Omega} x |\nabla v_k|^2 \, dx}{\int_{\Omega} |\nabla v_k|^2 \, dx} \to \frac{\So x^{(0)}}{\So} = x^{(0)}.
$$

\noindent Since the topology of $\rnn$ is normal, there exists an open set $V$ such that $\bar{\om} \subset V \subset \bar{V} \subset \subset U$. Because $F(u_k) \notin U$ and $|F(u_k) - F(v_k)| \to 0$ as $k \to \infty$, it follows that
$$
x^{(0)} \notin \bar{\Omega},
$$
yielding a contradiction. Thus, our claim is proved.

\smallskip
\noindent As $\Omega$ is smooth, we can choose an open neighborhood $U_0$ of $\bar{\Omega}$ such that each point $p \in U_0$ admits a unique nearest projection $q=\pi(p) \in \bar{\Omega}$, and the projection map $\pi : U_0 \to \bar{\Omega}$ is continuous.

\smallskip
\noindent Let $\delta > 0$ be chosen as above for this neighborhood $U_0$, and let $\Phi : \mathcal{M} \times [0,1] \to \mathcal{M}$ denote the deformation flow constructed previously. Define the map $h : \mathbb{S}^{N-1} \times [0,1] \to \bar{\Omega}$ by
$$
h(\sigma, t) := \pi\left(F\left(\Phi(v_t^{\sigma}, 1)\right)\right),
$$
where $v_t^{\sigma}$ is the aforementioned $2$-parameter family.

\smallskip
\noindent This map $h$ is  well-defined, continuous and in view of Lemma \ref{lemma:less_than_second_level} it satisfies
$$
\begin{aligned}
& h(\sigma, 0) = \pi\left(F\left(\Phi(v_0, 1)\right)\right) =: x^{(0)} \in \bar{\Omega}, \quad \text{for all } \sigma \in \mathbb{S}^{N-1}, \\
& h(\sigma, 1) = \sigma, \quad \text{for all } \sigma \in \mathbb{S}^{N-1}.
\end{aligned}
$$

\noindent Hence, $h$ defines a contraction of the unit sphere $\mathbb{S}^{N-1}$ into a point $x^{(0)}$ in $\bar{\Omega}$, contradicting \eqref{eq:domain_property_2}.\\[0pt]
\end{proof}

\section{Appendix} 

\begin{lemma}\label{signChangSolnEnergy}
Let $u$ be any sign-changing solution of the problem \eqref{Maineq} with $\lambda = 0$. Then the associated energy satisfies
$$
I_{0,s}(u) \geq \frac{2}{N}\So^{N/2},
$$
where $\So$ denotes the best constant in the Sobolev embedding (see Section~\ref{rem:propSnrecall}).
\end{lemma}
\begin{proof}
  Let $u$ be a sign-changing solution to the problem \eqref{Maineq} with $\lambda = 0$. Then $u$ satisfies the equation
\begin{equation}
    \begin{cases}
-\Delta u + (-\Delta)^s u = |u|^{2^*-2} u \quad \text{in } \Omega,\\
u=0\text{ in }\Rn\setminus\Omega,
\end{cases}
    \label{mainEq1}
\end{equation}
in the weak sense.
\noi
Taking $v = u^+ := \max\{u, 0\}$ as a test function in \eqref{mainEq1}, we obtain
\begin{equation*}
      \int_{\Omega} \nabla u \nabla u^+ {\rm d}x+ \iint_{\mathbb{R}^{2N}} \frac{(u(x) - u(y))(u^+(x) - u^+(y))}{|x - y|^{N+2s}}{\rm d}x \,{\rm d}y = \int_{\Omega} |u|^{2^*-2} uu^+{\rm d}x.
\end{equation*}
Expanding the above expansion we get,
\begin{equation} 
 \hspace{-0.3cm}\int_{\Omega}|\nabla u^{+}|^{2} {\rm d}x+\iint_{\mathbb{R}^{2N}} \frac{|u^{+}(x) - u^{+}(y)|^{2}}{|x-y|^{N+2s}}{\rm d}x \,{\rm d}y +2 \iint_{\mathbb{R}^{2N}} \frac{u^{+}(x) u^{-}(y)}{|x-y|^{N+2s}}{\rm d}x \,{\rm d}y=\int_{\Omega} |u^{+}|^{2^{*}}{\rm d}x \label{plusTestFunc}.
\end{equation}
Applying the Sobolev inequality to $u^+$ gives
\begin{equation*}
 \int_{\Omega} |\nabla u^{+}|^{2}{\rm d}x  \le\int_{\Omega} |u^{+}|^{2^{*}} {\rm d}x \le \left(\So^{-1} \int_{\Omega} |\nabla u^{+}|^{2}{\rm d}x \right)^{2^{*}/2}.
\end{equation*}
This yields 
\begin{equation}
\So^{N/2} \le \int_{\Omega} |\nabla u^{+}|^{2}{\rm d}x  \label{LowerBdOnEnergy1}.
\end{equation}
Now, taking $v = u^- := \max\{-u, 0\}$ as a test function in \eqref{mainEq1}, we obtain


\begin{equation}
- \int_{\Omega} |\nabla u^{-}|^2 \, {\rm d}x - \left[ u^{-} \right]_{s}^{2} - 2 \iint_{\mathbb{R}^{2N}} \frac{u^{+}(x) u^{-}(y)}{|x-y|^{N+2s}} \, {\rm d}x\,{\rm d}y = - \int_{\Omega} |u^-|^{2^{*}} \, {\rm d}x, \label{MinusTestFn}
\end{equation}

\noindent
Once again, by the Sobolev inequality, we obtain

\begin{equation}
\So^{N/2} \leq \int_{\Omega} |\nabla u^{-}|^2 \, {\rm d}x.\label{LowerBdEnergy2}
\end{equation}

\noi
Now consider the energy functional
\begin{align*}
&I_{0,s}(u) = \frac{1}{2} \int_{\Omega} |\nabla u|^2 \, {\rm d}x + \frac{1}{2} [u]_{s}^{2} - \frac{1}{2^{*}} \int_{\Omega} |u|^{2^{*}} \, {\rm d}x \\
&= \frac{1}{2} \int_{\Omega} |\nabla u^{+}|^{2} \, {\rm d}x + \frac{1}{2} \int_{\Omega} |\nabla u^{-}|^{2} \, {\rm d}x + \frac{1}{2} [u^{+}]_{s}^{2} + \frac{1}{2} [u^{-}]_{s}^{2} - \langle u^{+}, u^{-} \rangle_{s} \\
& \hspace{5cm}- \frac{1}{2^{*}} \int_{\Omega} |u^{+}|^{2^{*}} \, {\rm d}x - \frac{1}{2^{*}} \int_{\Omega} |u^{-}|^{2^{*}} \, {\rm d}x \\
&= \frac{1}{2} \int_{\Omega} |\nabla u^{+}|^{2} \, {\rm d}x + \frac{1}{2} [u^{+}]_{s}^{2} - \frac{1}{2^{*}} \int_{\Omega} |u^{+}|^{2^{*}} \, {\rm d}x + \frac{1}{2} \int_{\Omega} |\nabla u^{-}|^{2} \, {\rm d}x  \\
&\hspace{2cm}+ \frac{1}{2} [u^{-}]_{s}^{2} - \frac{1}{2^{*}} \int_{\Omega} |u^{-}|^{2^{*}} \, {\rm d}x +2 \iint_{\mathbb{R}^{2N}} \frac{u^{+}(x) u^{-}(y)}{|x-y|^{N+2s}} \, {\rm d}x\,{\rm d}y .
\end{align*}

\smallskip
\noindent
Using \eqref{plusTestFunc} and \eqref{MinusTestFn}, we substitute to get

\begin{align*}
I_{0,s}(u) &= \frac{1}{2} \int_{\Omega} |\nabla u^{+}|^{2} \, {\rm d}x + \frac{1}{2} [u^{+}]_{s}^{2} - \frac{1}{2^{*}} \left( \int_{\Omega} |\nabla u^{+}|^{2} \, {\rm d}x + [u^{+}]_{s}^{2} \right. \\
& \left.  + 2 \iint_{\mathbb{R}^{2N}} \frac{u^{+}(x) u^{-}(y)}{|x-y|^{N+2s}} \, {\rm d}x\,{\rm d}y \right)+ \frac{1}{2} \int_{\Omega} |\nabla u^{-}|^{2} \, {\rm d}x + \frac{1}{2} [u^{-}]_{s}^{2}  \\
&- \frac{1}{2^{*}} \left( \int_{\Omega} |\nabla u^{-}|^{2} \, {\rm d}x + [u^{-}]_{s}^{2}  + 2 \iint_{\mathbb{R}^{2N}} \frac{u^{+}(x) u^{-}(y)}{|x-y|^{N+2s}} \, {\rm d}x\,{\rm d}y \right)\\
& + 2 \iint_{\mathbb{R}^{2N}} \frac{u^{+}(x) u^{-}(y)}{|x-y|^{N+2s}} \, {\rm d}x\,{\rm d}y
\end{align*}

\begin{align*}
I_{0,s}(u) &= \frac{1}{N} \int_{\Omega} \left(|\nabla u^{+}|^{2}+ |\nabla u^{-}|^{2}\right) + \frac{1}{N} \left([u^{+}]_{s}^{2}+  [u^{-}]_{s}^{2}\right) + \frac{4}{N} \iint_{\mathbb{R}^{2N}} \frac{u^{+}(x) u^{-}(y)}{|x-y|^{N+2s}} \, {\rm d}x\,{\rm d}y.\\
\end{align*}

\noindent
Finally, using \eqref{LowerBdOnEnergy1} and \eqref{LowerBdEnergy2}, we conclude

\begin{equation*}
I_{0,s}(u) \geq \frac{2}{N} \So^{N/2}.
\end{equation*}
\end{proof}

\noindent{\bf Acknowledgment:} 
S. Chakraborty acknowledges the support of the Chair Professor grant (Project No. MI00148) at IIT Delhi, India. The author, Shammi Malhotra, is supported by the Prime Minister’s Research Fellowship (PMRF ID - 1403226), Diksha Gupta is supported by the PMRF fellowship (PMRF ID - 1401219).

\bibliographystyle{plain}
\bibliography{ref}

\end{document}